\documentclass[12pt,a4paper,reqno]{amsart}
\usepackage{hyperref}
\usepackage{cleveref}
\usepackage{amscd}
\usepackage{amsmath}
\usepackage{amsthm}
\usepackage{amsfonts}
\usepackage{amssymb}
\usepackage{verbatim}
\usepackage{enumerate}
\usepackage{xcolor}
\usepackage{graphicx}
\usepackage{siunitx}
\usepackage{bm}
\usepackage{mathtools}
\usepackage[tableposition=top]{caption}
\usepackage{booktabs,dcolumn}

\numberwithin{equation}{section}

\theoremstyle{plain}

\newtheorem{theorem}{Theorem}[section]
\newtheorem{proposition}[theorem]{Proposition}
\newtheorem{lemma}[theorem]{Lemma}
\newtheorem{corollary}[theorem]{Corollary}

\theoremstyle{definition}
\newtheorem{definition}[theorem]{Definition}
\newtheorem{remark}[theorem]{Remark}

\newtheorem*{acknowledgement}{Acknowledgements}

\begin{document}

\title[The Andoni--Naor--Neiman inequalities]{The Andoni--Naor--Neiman inequalities and \\isometric embeddability into a $\mathrm{CAT}(0)$ space}

\author{Tetsu Toyoda}
\address{Kogakuin University, 2665-1, Nakano, Hachioji, Tokyo, 192-0015 Japan}
\email{toyoda@cc.kogakuin.ac.jp}
\maketitle


\begin{abstract}
Andoni, Naor and Neiman (2018) established a family of 
quadratic metric inequalities that hold true in every $\mathrm{CAT}(0)$ space. 
As stated in their paper, this family seems to include all previously used quadratic metric inequalities that hold true in every $\mathrm{CAT}(0)$ space. 
We prove that there exists a metric space that satisfies all inequalities in this family 
but does not admit an isometric embedding into any $\mathrm{CAT}(0)$ space. 
More precisely, we prove that the $6$-point metric space constructed by Nina Lebedeva, 
which does not admit an isometric embedding into any $\mathrm{CAT}(0)$ space, 
satisfies all inequalities in this family. 
\end{abstract}

\section{Introduction}

We call a condition on a metric space $(X,d_{X})$ a 
{\em quadratic metric inequality on $n$ points} 
if there exists a real $n\times n$ matrix $(a_{ij})$ such that 
$(X,d_{X})$ satisfies the condition if and only if the inequality 
\begin{equation}\label{qm-ineq}
0\leq
\sum_{i=1}^n\sum_{j=1}^n a_{ij}d_X (x_i ,x_j)^2
\end{equation}
holds for any $x_1 ,\ldots ,x_n \in X$. 
We call a quadratic metric inequality a {\em $\mathrm{CAT}(0)$ quadratic metric inequality} if 
every $\mathrm{CAT}(0)$ space satisfies it. 
Andoni--Naor--Neiman \cite{ANN} proved the following theorem. 

\begin{theorem}[Andoni, Naor and Neiman \cite{ANN}]\label{ANN-th}
An $n$-point metric space $X$ admits an isometric embedding into a $\mathrm{CAT}(0)$ space if and only if $X$ 
satisfies all $\mathrm{CAT}(0)$ quadratic metric inequalities on $n$ points. 
\end{theorem}

For the original statement of Theorem~\ref{ANN-th} in full generality, see \cite[Proposition~3]{ANN}. 
To find a characterization of those metric spaces that admit 
an isometric embedding into a $\mathrm{CAT}(0)$ space is a longstanding open problem stated by 
Gromov in \cite[\S15]{Gr2} and \cite[Section 1.19+]{Gr1} 
(see also \cite[Section 7]{AKP}, \cite[Chapter 10, E]{AKP-book}, \cite[Section 1.4]{ANN}, \cite[Section 2.2]{Bac} 
and \cite[Section 1]{toyoda-five}). 
Theorem~\ref{ANN-th} tells us that we can answer this problem if 
we can characterize $\mathrm{CAT}(0)$ quadratic metric inequalities. 
To see some recent developments in this direction, 
we recall the following family of $\mathrm{CAT}(0)$ quadratic metric inequalities on $4$ points. 

\begin{definition}
We say that a metric space $(X,d_X )$ satisfies 
the {\em $\boxtimes$-inequalities} if we have 
\begin{multline*}
0
\leq
(1-s)(1-t) d_X (x,y)^2 +s(1-t) d_X (y,z)^2 +st d_X (z,w)^2 +(1-s)t d_X (w,x)^2 \\
-s(1-s) d_X (x,z)^2 -t(1-t) d_X (y,w)^2
\end{multline*}
for any $s, t \in\lbrack0,1\rbrack$ and any $x,y,z,w\in X$.  
\end{definition}
Gromov \cite{Gr2} and Sturm \cite{St} proved independently that every $\mathrm{CAT}(0)$ space satisfies the 
$\boxtimes$-inequalities. 
The name ``$\boxtimes$-inequalities'' is based on a notation used by Gromov \cite{Gr2}. 
Sturm \cite{St} called these inequalities the {\em weighted quadruple inequalities}. 
The present author \cite{toyoda-five} proved the following theorem. 

\begin{theorem}[\cite{toyoda-five}]\label{five-th}
If a metric space $X$ satisfies the $\boxtimes$-inequalities, 
then $X$ satisfies all $\mathrm{CAT}(0)$ quadratic metric inequalities on $5$ points. 
\end{theorem}

It follows from Theorem~\ref{ANN-th} and Theorem~\ref{five-th} that a metric space $X$ with $|X|\leq 5$ admits 
an isometric embedding into a $\mathrm{CAT}(0)$ space if and only if $X$ satisfies the 
$\boxtimes$-inequalities. 
On the other hand, Nina Lebedeva constructed a $6$-point metric space that satisfies 
the $\boxtimes$-inequalities but does not admit an isometric embedding into any $\mathrm{CAT}(0)$ space 
(see the arXiv version of \cite[\S 7.2]{AKP} or Theorem \ref{Lebedeva-th} in this paper). 
Therefore, the validity of the $\boxtimes$-inequalities is not a sufficient condition for even a $6$-point metric space 
to admit an isometric embedding into a $\mathrm{CAT}(0)$ space.

Toward a characterization of $\mathrm{CAT}(0)$ quadratic metric inequalities, 
Andoni--Naor--Neiman \cite{ANN} proposed a general procedure to obtain $\mathrm{CAT}(0)$ quadratic metric inequalities, 
and presented the following large family of $\mathrm{CAT}(0)$ quadratic metric inequalities (see Lemma~27 and Section 5.1 in \cite{ANN}). 
Throughout this paper, 
we denote by $\lbrack n\rbrack$ the set $\{ 1,2,\ldots ,n\}$ for each positive integer $n$. 

\begin{theorem}[Andoni, Naor and Neiman \cite{ANN}]\label{ANN-ineq-th}
Fix a positive integer $n$. 
Let $m$ be a positive integer, and let $c_1 ,\ldots ,c_m$ be positive real numbers. 
Suppose for every $k\in\lbrack m\rbrack$, $p_1^k ,\ldots ,p_n^k ,q_1^k ,\ldots ,q_n^k$ are positive real numbers that satisfy 
\begin{equation*}
\sum_{i=1}^n p_i^k =\sum_{j=1}^n q_j^k =1.
\end{equation*}
Suppose for every $k\in\lbrack m\rbrack$, 
$A_k=(a_{ij}^k)$ and $B_k=(b_{ij}^k)$ are $n\times n$ matrices with nonnegative real entries 
that satisfy 
\begin{equation*}
\sum_{s=1}^n a_{is}^{k}+\sum_{s=1}^n b_{sj}^{k}=p_i^k +q_j^k
\end{equation*}
for any $i,j\in\lbrack n\rbrack$. 
Then we have 
\begin{equation}\label{ANN-ineq}
\sum_{k=1}^{m}\hspace{2mm}\sum_{i,j\in\lbrack n\rbrack : a_{ij}^k +b_{ij}^k >0}\frac{c_k a_{ij}^k b_{ij}^k}{a_{ij}^k +b_{ij}^k}d_X (x_i ,x_j )^2
\leq
\sum_{k=1}^{m}\sum_{i=1}^n \sum_{j=1}^n c_k p_i^k q_j^k d_X (x_i ,x_j )^2
\end{equation}
for any $\mathrm{CAT}(0)$ space $(X,d_X)$ and any points $x_1 ,\ldots ,x_n \in X$. 
\end{theorem}

It is easily seen that many known $\mathrm{CAT}(0)$ quadratic metric inequalities, 
including the $\boxtimes$-inequalities, are of the form \eqref{ANN-ineq}. 
Moreover, as stated in \cite[Section 5.1]{ANN}, 
it seems that all of the previously used $\mathrm{CAT}(0)$ quadratic metric inequalities 
are of the form \eqref{ANN-ineq}. 
So it is natural to ask whether the inequalities of the form \eqref{ANN-ineq} capture the totality of $\mathrm{CAT}(0)$ quadratic metric inequalities 
(cf. \cite[Question 31]{ANN}). 
In this paper, we answer this question negatively by proving the following theorem, which is our main result. 

\begin{theorem}\label{main-th}
There exists a $6$-point metric space $(X,d_X )$ that satisfies the following conditions: 
\begin{enumerate}
\item[$\mathrm{(i)}$]
For every positive integer $n$, any points $x_1 ,\ldots ,x_n \in X$ satisfy the inequality 
\begin{equation}\label{main-th-pi-ineq}
\frac{1}{2}\sum_{i,j\in\lbrack n\rbrack : \pi_{ij}+\pi_{ji}>0}
\frac{\pi_{ij}\pi_{ji}}{\pi_{ij}+\pi_{ji}}d_X (x_i ,x_j)^2
\leq
\sum_{i=1}^{n}\sum_{j=1}^{n}p_i q_j d_X(x_i ,x_j)^2
\end{equation}
for any nonnegative real numbers $p_1 ,\ldots ,p_n ,q_1 ,\ldots ,q_n$ and any 
$n\times n$ matrix $\left(\pi_{ij}\right)$ with nonnegative real entries that satisfy 
\begin{equation*}
\sum_{k=1}^n p_k =\sum_{k=1}^n q_k =1,\quad
\sum_{k=1}^n \pi_{ik}=p_i+q_i
\end{equation*}
for every $i\in\lbrack n\rbrack$. 
\item[$\mathrm{(ii)}$]
$X$ does not admit an isometric embedding into any $\mathrm{CAT}(0)$ space. 
\end{enumerate}
\end{theorem}

As we will show in Section \ref{reformulation-sec}, the validity of all inequalities of the form \eqref{main-th-pi-ineq} 
implies the validity of all inequalities of the form \eqref{ANN-ineq}. 
Thus Theorem~\ref{main-th} implies the following corollary.

\begin{corollary}\label{main-coro}
There exists a $6$-point metric space $(X,d_X )$ that satisfies the following conditions: 
\begin{itemize}
\item[$\mathrm{(i)}$]
For every positive integer $n$, any points $x_1 ,\ldots ,x_n \in X$ satisfy the inequality \eqref{ANN-ineq} for 
any $m$, $c_1 ,\ldots ,c_m$, $\left( p_{i}^{1}\right)_{i\in\lbrack n\rbrack},\ldots ,\left( p_{i}^{m}\right)_{i\in\lbrack n\rbrack}$, 
$\left( q_{j}^{1}\right)_{j\in\lbrack n\rbrack},\ldots ,\left( q_{j}^{m}\right)_{j\in\lbrack n\rbrack}$, 
$A_1 ,\ldots ,A_m$, $B_1 ,\ldots ,B_m$ as in the statement of Theorem~\ref{ANN-ineq-th}. 
\item[$\mathrm{(ii)}$]
$X$ does not admit an isometric embedding into any $\mathrm{CAT}(0)$ space. 
\end{itemize}
\end{corollary}

More precisely, we will prove that 
Lebedeva's $6$-point metric space 
that we have mentioned before satisfies the condition $\mathrm{(i)}$ in Theorem~\ref{main-th} and 
the condition $\mathrm{(i)}$ in Corollary~\ref{main-coro}. 

\begin{remark}
As stated in \cite[Section 5.1]{ANN}, 
no $\mathrm{CAT}(0)$ quadratic metric inequality that is not a necessary condition 
for a metric space to satisfy all $\mathrm{CAT}(0)$ quadratic metric inequalities of the form \eqref{ANN-ineq} 
was known explicitly when the first version of this paper was written. 
Moreover, to the best of the author's knowledge, at that time, no $\mathrm{CAT}(0)$ quadratic metric inequality 
that is not a necessary condition for a metric space to satisfy even the $\boxtimes$-inequalities was known explicitly, 
although the existence of such inequalities was guaranteed by the construction of Lebedeva's $6$-point metric space. 
After the first version of this paper was written, the present author \cite{toyoda-six} found a $\mathrm{CAT}(0)$ quadratic metric inequality 
that fails in Lebedeva's $6$-point metric space. 
This provides the first explicit example of a $\mathrm{CAT}(0)$ quadratic metric inequality 
that is not a necessary condition for a metric space to satisfy the $\boxtimes$-inequalities. 
Moreover, when combined with the result of the present paper, the inequality found in \cite{toyoda-six} 
turned out to be the first explicit example of a $\mathrm{CAT}(0)$ quadratic metric inequality 
that is not a necessary condition for a metric space to satisfy all $\mathrm{CAT}(0)$ quadratic metric inequalities of the form \eqref{ANN-ineq} as well 
(see \cite[Corollary 1.5]{toyoda-six}).
\end{remark}

\subsection{Organization of the paper}
This paper is organized as follows. 
In Section \ref{reformulation-sec}, we reformulate and generalize 
the Andoni--Naor--Neiman inequalities \eqref{ANN-ineq}, 
and prove that the validity of all inequalities of the form \eqref{main-th-pi-ineq} 
implies the validity of all inequalities of the form \eqref{ANN-ineq}. 
In Section \ref{Lebedeva-sec}, we recall the definition of Lebedeva's $6$-point metric space. 
In Section \ref{proof-sec}, we prove that 
Lebedeva's $6$-point metric space 
satisfies the condition $\mathrm{(i)}$ in Theorem~\ref{main-th}, 
which proves Theorem~\ref{main-th}. 

\section{A reformulation of the Andoni--Naor--Neiman inequalities}\label{reformulation-sec}

In this section, we slightly reformulate and generalize the Andoni--Naor--Neiman inequalities \eqref{ANN-ineq}. 
We first recall the definition of a $\mathrm{CAT}(0)$ space. 

\begin{definition}\label{CAT(0)-def}
A complete metric space $(X,d_X )$ is called a {\em $\mathrm{CAT}(0)$ space} if 
for any $x,y\in X$ and any $t\in\lbrack 0,1\rbrack$, there exists $z\in X$ that satisfies 
\begin{equation}\label{CAT(0)-ineq}
d_X (w,z)^2
\leq
(1-t)d_X (w,x)^2 +td_X (w,y)^2 -t(1-t)d_X (x,y)^2
\end{equation}
for any $w\in X$. 
\end{definition}

\begin{remark}
If $d_X$ is a semimetric on $X$, 
or in other words, if $(X,d_X )$ satisfies the axioms of a metric space 
except for the requirement that $d_X$ satisfies the triangle inequality, 
then for any $x,y,w\in X$, the triangle inequality 
\begin{equation*}
d_X  (x,y) \leq d_X (w,x)+d_X (w,y)
\end{equation*}
holds if and only if the inequality 
\begin{equation}\label{sq-tri-ineq}
t(1-t)d_X (x,y)^2 \leq (1-t)d_X (w,x)^2 +td_X (w,y)^2
\end{equation}
holds for any $t\in\lbrack 0,1\rbrack$. 
We remark that \eqref{sq-tri-ineq} follows from \eqref{CAT(0)-ineq}. 
\end{remark}

Let $(X,d_X )$ be a $\mathrm{CAT}(0)$ space, and let $x_1 ,\ldots ,x_n \in X$. 
Suppose $p_1 ,\ldots ,p_n$ are positive real numbers with $\sum_{i=1}^n p_i =1$. 
Then there exists a unique point $z\in X$ that satisfies 
\begin{equation}\label{variance-ineq}
d_X (w ,z)^2 +\sum_{i=1}^n p_i d_X (z, x_i )^2
\leq
\sum_{i=1}^n p_i d_X (w, x_i )^2
\end{equation}
for any $w\in X$ (see Lemma~4.4 and Theorem~4.9 in \cite{St}). 
The point $z\in X$ with this property is called the {\em barycenter} of the probability measure $\mu =\sum_{i=1}^n p_i \delta_{x_i}$ on $X$, 
and denoted by $\mathrm{bar}(\mu)$, 
where 
$\delta_{x_i}$ is the Dirac measure at $x_i$. 
It is easily seen that 
the point $z$ in Definition \ref{CAT(0)-def} is no other than $\mathrm{bar}\left( (1-t)\delta_x +t\delta_y\right)$. 
 
Although the proof of the following proposition is just a slight reformulation of the proof of Lemma~27 in \cite{ANN}, 
this proposition will play a key role in the proof of Theorem~\ref{main-th}. 
We denote by $\binom{\lbrack n\rbrack}{2}$ the set of all $2$-element subsets of $\lbrack n\rbrack$.

\begin{proposition}[cf. Lemma~27 in \cite{ANN}]\label{pi-prop}
Fix a positive integer $n$. 
Suppose $p_1 ,\ldots ,p_n$, $q_1 ,\ldots ,q_n$ are 
nonnegative real numbers 
that satisfy $\sum_{i=1}^n p_i =\sum_{j=1}^n q_j =1$. 
Suppose $\left(\pi_{ij}\right)$ is an $n\times n$ matrix with nonnegative real entries 
that satisfies 
\begin{equation*}
\sum_{k=1}^n \pi_{ik}=p_i+q_i
\end{equation*}
for every $i\in\lbrack n\rbrack$. 
Suppose $(X,d_X)$ is a $\mathrm{CAT}(0)$ space and $x_1 ,\ldots ,x_n \in X$. 
We set 
\begin{equation*}
z=\mathrm{bar}\left(\sum_{i=1}^n p_i \delta_{x_i}\right) ,\quad
z_{ij}=\mathrm{bar}\left(\frac{\pi_{ij}}{\pi_{ij}+\pi_{ji}}\delta_{x_i} +\frac{\pi_{ji}}{\pi_{ij}+\pi_{ji}}\delta_{x_j}\right)
\end{equation*}
for any $i,j\in\lbrack n\rbrack$ with $\pi_{ij}+\pi_{ji}>0$. 
Then we have 
\begin{multline*}
\sum_{i,j\in\lbrack n\rbrack : \pi_{ij}+\pi_{ji}>0}\pi_{ij}d_{X}(z,z_{ij})^2
+
\frac{1}{2}\sum_{i,j\in\lbrack n\rbrack : \pi_{ij}+\pi_{ji}>0}
\frac{\pi_{ij}\pi_{ji}}{\pi_{ij}+\pi_{ji}}d_X (x_i ,x_j)^2 \\
\leq
\sum_{i=1}^{n}\sum_{j=1}^{n}p_i q_j d_X(x_i ,x_j)^2 .
\end{multline*}
\end{proposition}

\begin{proof}
It follows from the inequality \eqref{variance-ineq} that  
\begin{equation*}
d_X (x_j ,z)^2 +\sum_{i=1}^n p_i d_X (x_i ,z)^2
\leq
\sum_{i=1}^n p_i d_X (x_i ,x_j)^2
\end{equation*}
holds for each $j\in\lbrack n\rbrack$. 
By multiplying this inequality by $q_j$ and summing over $j\in\lbrack n\rbrack$, we obtain 
\begin{equation*}
\sum_{j=1}^n q_j d_X (x_j ,z)^2 +\sum_{i=1}^n p_i d_X (x_i ,z)^2
\leq
\sum_{i=1}^n \sum_{j=1}^n p_i q_j d_X (x_i ,x_j)^2 .
\end{equation*}
Since we have $z_{ij}=z_{ji}$ for any $i,j\in\lbrack n\rbrack$ with $\pi_{ij}+\pi_{ji}>0$, 
it follows from 
\eqref{CAT(0)-ineq} that 
\begin{align*}
&\sum_{j=1}^{n}q_j d_X (x_j ,z)^2 +\sum_{i=1}^n p_i d_X (x_i ,z)^2
=
\sum_{k=1}^{n}(p_k +q_k )d_X (x_k ,z)^2
=
\sum_{k=1}^{n}\sum_{l=1}^{n}\pi_{kl}d_X (x_k ,z)^2 \\
&\hspace{10mm}=
\sum_{k=1}^{n}\pi_{kk}d_{X}(x_{k} ,z)^2 \\
&\hspace{20mm}
+\sum_{\{k,l\}\in\binom{\lbrack n\rbrack}{2} : \pi_{kl} +\pi_{lk}> 0}
(\pi_{kl}+\pi_{lk})
\left(\frac{\pi_{kl}}{\pi_{kl}+\pi_{lk}}d_X (x_k ,z)^2 +\frac{\pi_{lk}}{\pi_{kl}+\pi_{lk}}d_X (x_l ,z)^2 \right) \\
&\hspace{10mm}\geq
\sum_{k\in\lbrack n\rbrack : \pi_{kk}> 0}\pi_{kk}d_{X}(z_{kk} ,z)^2 \\
&\hspace{20mm}
+\sum_{\{k,l\}\in\binom{\lbrack n\rbrack}{2} : \pi_{kl} +\pi_{lk}> 0}
(\pi_{kl}+\pi_{lk})
\left(d_X (z,z_{kl})^2 +\frac{\pi_{kl}\pi_{lk}}{(\pi_{kl}+\pi_{lk})^2}d_X (x_k ,x_l)^2 \right) \\
&\hspace{10mm}=
\sum_{i,j\in\lbrack n\rbrack : \pi_{ij}+\pi_{ji}> 0}\pi_{ij}d_{X}(z,z_{ij})^2
+
\frac{1}{2}\sum_{i,j\in\lbrack n\rbrack : \pi_{ij}+\pi_{ji}> 0}
\frac{\pi_{ij}\pi_{ji}}{\pi_{ij}+\pi_{ji}}d_X (x_i ,x_j)^2 ,
\end{align*}
which proves the proposition. 
\end{proof}

\begin{remark}
Because the statement of Proposition~\ref{pi-prop} is symmetric with respect to $(p_1 ,\ldots ,p_n )$ and $(q_1 ,\ldots ,q_n )$ 
except for the definition of $z$, the same statement holds true 
if we set $z=\mathrm{bar}(\sum_{j=1}^n q_j \delta_{x_j})$ 
instead of $z=\mathrm{bar}(\sum_{i=1}^n p_i \delta_{x_i})$. 
\end{remark}

The following corollary follows immediately from Proposition~\ref{pi-prop}. 

\begin{corollary}[cf. Lemma~27 in \cite{ANN}]\label{pi-coro}
Fix a positive integer $n$. 
Suppose $p_1 ,\ldots ,p_n$, $q_1 ,\ldots ,q_n$, and 
$(\pi_{ij})$ are as in the statement of Proposition~\ref{pi-prop}. 
Suppose $(X,d_X)$ is a $\mathrm{CAT}(0)$ space and $x_1 ,\ldots ,x_n \in X$. 
Then we have 
\begin{equation}\label{pi-ineq}
\frac{1}{2}\sum_{i,j\in\lbrack n\rbrack : \pi_{ij}+\pi_{ji}> 0}
\frac{\pi_{ij}\pi_{ji}}{\pi_{ij}+\pi_{ji}}d_X (x_i ,x_j)^2
\leq
\sum_{i=1}^{n}\sum_{j=1}^{n}p_i q_j d_X(x_i ,x_j)^2 .
\end{equation}
\end{corollary}

We define the following condition. 

\begin{definition}
Fix a positive integer $n$. 
Let $(X,d_X )$ be a semimetric space. 
We say that $X$ satisfies the {\em $\mathrm{ANN}(n)$ inequalities} if 
any points $x_1 ,\ldots ,x_n \in X$ satisfy the 
inequality \eqref{pi-ineq} for any $p_1 ,\ldots ,p_n ,q_1 ,\ldots ,q_n$ and 
$(\pi_{ij})$ as in the statement of Proposition~\ref{pi-prop}. 
\end{definition}

\begin{remark}
It is easily seen that for every $t\in\lbrack 0,1\rbrack$, the inequality of the form \eqref{sq-tri-ineq} belongs to 
the $\mathrm{ANN}(n)$ inequalities for any $n\geq 3$. 
Therefore, a semimetric space is a metric space whenever it satisfies the $\mathrm{ANN}(n)$ inequalities for some integer $n\geq 3$. 
\end{remark}

To set up some basic facts about the $\mathrm{ANN}(n)$ inequalities, 
we recall the following elementary fact. 

\begin{proposition}\label{elementary-prop}
Let $n$ be a positive integer. 
Suppose $a_1,\ldots,a_n$ and $b_1,\ldots,b_n$ are nonnegative real numbers 
with $a_i+b_i>0$ for every $i\in\lbrack n\rbrack$. 
Then 
\begin{equation*}
\sum_{i=1}^{n}\frac{a_i b_i}{a_i +b_i}
\leq
\frac{\left(\sum_{i=1}^{n}a_{i}\right)
\left(\sum_{i=1}^{n}b_{i}\right)}{\left(\sum_{i=1}^{n}a_{i}\right) +\left(\sum_{i=1}^{n}b_{i}\right)} .
\end{equation*}
\end{proposition}

\begin{proof}
We prove the proposition by induction on $n$. The case $n=1$ is trivial. 
Assume it holds for some positive integer $n$. 
Let $a_1,\ldots,a_{n+1}$ and $b_1,\ldots,b_{n+1}$ be nonnegative real numbers 
with $a_i+b_i>0$ for every $i\in\lbrack n+1\rbrack$. 
Set $a=\sum_{i=1}^{n}a_i$ and $b=\sum_{i=1}^{n}b_{i}$. 
A straightforward calculation shows that 
\begin{align*}
&\ \left(a+a_{n+1}\right)\left(b+b_{n+1}\right)\left(a+b\right)\left(a_{n+1}+b_{n+1}\right) \\
&\quad
-
ab\left(a_{n+1}+b_{n+1}\right)\left(a+a_{n+1}+b+b_{n+1}\right)
-
a_{n+1}b_{n+1}\left(a+b\right)\left(a+a_{n+1}+b+b_{n+1}\right) \\
&\qquad
=
\left(ab_{n+1}-a_{n+1}b\right)^2
\geq
0 ,
\end{align*}
which implies that 
\begin{equation}\label{two-terms-ineq}
\frac{ab}{a+b}+\frac{a_{n+1}b_{n+1}}{a_{n+1}+b_{n+1}}
\leq
\frac{\left(a+a_{n+1}\right)\left(b+b_{n+1}\right)}{a+a_{n+1}+b+b_{n+1}}.
\end{equation}
By the inductive hypothesis, we have 
\begin{equation}\label{inductive-hyp-ineq}
\sum_{i=1}^{n}\frac{a_i b_i}{a_i +b_i}
\leq
\frac{ab}{a+b}.
\end{equation}
Combining \eqref{two-terms-ineq} and \eqref{inductive-hyp-ineq}, we obtain
\begin{equation*}
\sum_{i=1}^{n+1}\frac{a_i b_i}{a_i +b_i}
\leq
\frac{\left(\sum_{i=1}^{n+1}a_{i}\right)
\left(\sum_{i=1}^{n+1}b_{i}\right)}{\left(\sum_{i=1}^{n+1}a_{i}\right)+\left(\sum_{i=1}^{n+1}b_{i}\right)} ,
\end{equation*}
which completes the proof. 
\end{proof}

We will use the following proposition in Section \ref{proof-sec}. 

\begin{proposition}\label{n-points-prop}
Let $n$ be a positive integer, and let 
$(X=\{ x_{1} ,\ldots ,x_{n} \} ,d_X )$ be a semimetric space with $|X|=n$. 
Assume that the inequality 
\begin{equation*}
\frac{1}{2}\sum_{i,j\in\lbrack n\rbrack : \pi_{ij}+\pi_{ji}>0}\frac{\pi_{ij}\pi_{ji}}{\pi_{ij}+\pi_{ji}}d_X (x_i ,x_j)^2
\leq
\sum_{i=1}^{n}\sum_{j=1}^{n}p_i q_j d_X(x_i ,x_j)^2
\end{equation*}
holds true for any $p_1 ,\ldots ,p_n ,q_1 ,\ldots ,q_n$, and 
$(\pi_{ij})$ as in the statement of Proposition~\ref{pi-prop}. 
Then $X$ satisfies the $\mathrm{ANN}(m)$ inequalities for all positive integers $m$. 
\end{proposition}

\begin{proof}
Fix a positive integer $m$. 
Fix nonnegative real numbers $p_1 ,\ldots ,p_{m}$, $q_{1},\ldots ,q_{m}$ and 
an $m\times m$ matrix 
$(\pi_{ij})$ with nonnegative real entries such that 
\begin{equation*}
\sum_{k=1}^{m}p_k =\sum_{k=1}^{m}q_k =1,\quad
\sum_{k=1}^{m}\pi_{ik}=p_i +q_i
\end{equation*}
for every $i\in\lbrack m\rbrack$. 
Choose a map $\varphi :\lbrack m\rbrack\to\lbrack n\rbrack$ arbitrarily. 
For each $k,l\in\lbrack n\rbrack$, we set 
\begin{equation*}
\tilde{p}_{k}
=
\sum_{i\in\varphi^{-1}(\{ k\})}p_{i},\qquad
\tilde{q}_{k}
=
\sum_{i\in\varphi^{-1}(\{ k\})}q_{i},\qquad
\tilde{\pi}_{kl}
=
\sum_{(i,j)\in\varphi^{-1}(\{ k\})\times\varphi^{-1}(\{ l\})}\pi_{ij}.
\end{equation*}
Then it is easily seen that we have 
\begin{equation*}
\sum_{k=1}^n \tilde{p}_k
=
\sum_{k=1}^n \tilde{q}_k
=
1,\quad
\sum_{k=1}^n \tilde{\pi}_{ik}=\tilde{p}_i +\tilde{q}_i
\end{equation*}
for every $i\in\lbrack n\rbrack$. 
It follows from Proposition~\ref{elementary-prop} and the hypothesis that 
\begin{align*}
&\frac{1}{2}\sum_{i,j\in\lbrack m\rbrack : \pi_{ij}+\pi_{ji}> 0}\frac{\pi_{ij}\pi_{ji}}{\pi_{ij}+\pi_{ji}}d_X (x_{\varphi (i)},x_{\varphi (j)})^2 \\
&\quad =
\frac{1}{2}\sum_{k,l\in\lbrack n\rbrack : \tilde{\pi}_{kl}+\tilde{\pi}_{lk}> 0}
\left(
\sum_{(i,j)\in\varphi^{-1}(\{ k\})\times\varphi^{-1}(\{ l\}) : \pi_{ij}+\pi_{ji}>0}\frac{\pi_{ij}\pi_{ji}}{\pi_{ij}+\pi_{ji}}\right) d_X (x_{k},x_{l})^2 \\
&\quad\leq
\frac{1}{2}\sum_{k,l\in\lbrack n\rbrack : \tilde{\pi}_{kl}+\tilde{\pi}_{lk}> 0}
\left(
\frac{\tilde{\pi}_{kl}\tilde{\pi}_{lk}}{\tilde{\pi}_{kl}+\tilde{\pi}_{lk}}\right) d_X (x_{k},x_{l})^2 \\
&\quad\leq
\sum_{k=1}^{n}\sum_{l=1}^{n}\tilde{p}_{k}\tilde{q}_{l} d_{X}(x_{k},x_{l})^2 \\
&\quad =
\sum_{k=1}^{n}\sum_{l=1}^{n}\left(\sum_{i\in\varphi^{-1}(\{ k\})}p_{i}\right)\left(\sum_{j\in\varphi^{-1}(\{ l\})}q_{j}\right)
d_{X}(x_{k},x_{l})^2 \\
&\quad =
\sum_{k=1}^{n}\sum_{l=1}^{n}
\left(\sum_{i\in\varphi^{-1}(\{ k\})}\sum_{j\in\varphi^{-1}(\{ l\})}p_i q_j d_{X}(x_{\varphi (i)},x_{\varphi (j)})^2 \right) \\
&\quad =
\sum_{i=1}^{m}\sum_{j=1}^{m}
p_i q_j d_{X}(x_{\varphi (i)},x_{\varphi (j)})^2 ,
\end{align*}
which proves the proposition. 
\end{proof}

The following corollary follows from Proposition~\ref{n-points-prop} immediately. 

\begin{corollary}\label{n-points-coro}
Suppose $m$ and $n$ are positive integers with $m\leq n$, and 
$X$ is a semimetric space that satisfies the $\mathrm{ANN}(n)$ inequalities. 
Then $X$ satisfies the $\mathrm{ANN}(m)$ inequalities. 
\end{corollary}

The following lemma clarifies the relation between the $\mathrm{ANN}(n)$ inequalities and 
the inequalities of the form \eqref{ANN-ineq} established by Andoni--Naor--Neiman \cite{ANN}. 

\begin{lemma}\label{reformulation-lemma}
Fix a positive integer $n$ and 
nonnegative real numbers $p_1 ,\ldots ,p_n$, $q_1 ,\ldots ,q_n$ with 
$\sum_{i=1}^n p_i =\sum_{j=1}^n q_j =1$. 
Let $(X,d_X )$ be a semimetric space. 
Suppose $x_1 ,\ldots ,x_n \in X$ are points that satisfy the inequality 
\eqref{pi-ineq} for any $n\times n$ 
matrix $\left(\pi_{ij}\right)$ with nonnegative real entries 
satisfying $\sum_{k=1}^n \pi_{ik}=p_i+q_i$ for all $i\in\lbrack n\rbrack$. 
Then, 
for any $n\times n$ matrices $A=(a_{ij})$ and $B=(b_{ij})$ with nonnegative real entries satisfying 
$\sum_{k=1}^n a_{ik}+\sum_{k=1}^n b_{kj}=p_i +q_j$ 
for all $i,j\in\lbrack n\rbrack$, 
the inequality 
\begin{equation*}
\sum_{i,j\in\lbrack n\rbrack : a_{ij}+b_{ij}>0}\frac{a_{ij}b_{ij}}{a_{ij}+b_{ij}}d_X (x_i ,x_j )^2
\leq
\sum_{i=1}^n \sum_{j=1}^n  p_i q_j d_X (x_i ,x_j )^2
\end{equation*}
holds. 
\end{lemma}

\begin{proof}
Fix $n\times n$ matrices $A=(a_{ij})$ and $B=(b_{ij})$ with nonnegative real entries that satisfy 
\begin{equation*}
\sum_{k=1}^n a_{ik}+\sum_{k=1}^n b_{kj}=p_i +q_j
\end{equation*}
for any $i,j\in\lbrack n\rbrack$. 
Set $\pi_{ij}=a_{ij}+b_{ji}$ for any $i,j\in\lbrack n\rbrack$. 
Then we have 
$\sum_{k=1}^n \pi_{ik}=p_i+q_i$ for every $i\in\lbrack n\rbrack$ clearly. 
If $a_{ij}+b_{ij}>0$ and $a_{ji}+b_{ji}>0$, then we have 
\begin{equation*}
\frac{a_{ij}b_{ij}}{a_{ij}+b_{ij}}+\frac{a_{ji}b_{ji}}{a_{ji}+b_{ji}}
\leq
\frac{(a_{ij}+b_{ji})(b_{ij}+a_{ji})}{(a_{ij}+b_{ji})+(b_{ij}+a_{ji})}
=
\frac{\pi_{ij}\pi_{ji}}{\pi_{ij}+\pi_{ji}}
\end{equation*}
by Proposition~\ref{elementary-prop}. 
If $a_{ij}+b_{ij}>0$ and $a_{ji}+b_{ji}=0$, then 
we have 
$\pi_{ij}=a_{ij}$ and $\pi_{ji}=b_{ij}$. 
If $a_{ij}+b_{ij}=0$ and $a_{ji}+b_{ji}>0$, then 
we have 
$\pi_{ij}=b_{ji}$ and $\pi_{ji}=a_{ji}$. 
Therefore, we have 
\begin{equation*}
\sum_{i,j\in\lbrack n\rbrack : a_{ij}+b_{ij}>0}\frac{a_{ij}b_{ij}}{a_{ij}+b_{ij}}d_X (x_i ,x_j )^2
\leq
\frac{1}{2}\sum_{i,j\in\lbrack n\rbrack : \pi_{ij}+\pi_{ji}>0}
\frac{\pi_{ij}\pi_{ji}}{\pi_{ij}+\pi_{ji}}d_X (x_i ,x_j)^{2} .
\end{equation*}
Combining this with the hypothesis that the inequality \eqref{pi-ineq} holds true, we obtain the desired inequality. 
\end{proof}

The following corollary is an immediate consequence of Lemma~\ref{reformulation-lemma} and Corollary~\ref{n-points-coro}. 

\begin{corollary}\label{pi-to-AB-coro}
Let $n_0$ be a positive integer, and let $n\in\lbrack n_0 \rbrack$. 
Let $X$ be a semimetric space that satisfies the $\mathrm{ANN}(n_0 )$ inequalities. 
Then any $x_1 ,\ldots ,x_n \in X$ satisfy the inequality \eqref{ANN-ineq} for 
any $m$, $c_1 ,\ldots ,c_m$, $\left( p_{i}^{1}\right)_{i\in\lbrack n\rbrack},\ldots ,\left( p_{i}^{m}\right)_{i\in\lbrack n\rbrack}$, 
$\left( q_{j}^{1}\right)_{j\in\lbrack n\rbrack},\ldots ,\left( q_{j}^{m}\right)_{j\in\lbrack n\rbrack}$, 
$A_1 ,\ldots ,A_m$, $B_1 ,\ldots ,B_m$ as in the statement of Theorem~\ref{ANN-ineq-th}. 
\end{corollary}

\section{Lebedeva's 6-point metric space}\label{Lebedeva-sec}

In this section, we recall the $6$-point metric space constructed by Nina Lebedeva, 
which appeared in the arXiv version of \cite[\S 7.2]{AKP}. 

For any two points $a$ and $b$ in a Euclidean space, 
we denote by $\lbrack a,b\rbrack$ the line segment joining $a$ and $b$, and by $(a,b)$ the 
set $\lbrack a,b\rbrack\setminus\{ a,b\}$. 
We denote by $\mathrm{conv}(S)$ the convex hull of a subset $S$ of the $3$-dimensional Euclidean space $\mathbb{R}^3$. 
Suppose $x_1 ,\ldots ,x_6$ are six distinct points in 
$\mathbb{R}^3$ that satisfy the following conditions: 
\begin{align}
|(x_1 ,x_3)\cap (x_2, x_4)|&=1;\label{lebedeva-condition1}\\
\left|(x_5 ,x_6)\cap\left(\mathrm{conv}(\{ x_1 ,x_2 ,x_3 ,x_4\})\setminus
\cup_{\{ i,j\}\in\binom{\lbrack 4\rbrack}{2}}\lbrack x_i ,x_j\rbrack
\right)\right| &=1. \label{lebedeva-condition2}
\end{align}
It follows from the conditions \eqref{lebedeva-condition1} and \eqref{lebedeva-condition2} that  
the points $x_{1},\ldots ,x_{4}$ form the vertices of a convex quadrilateral in a $2$-dimensional affine subspace of $\mathbb{R}^3$, 
and that the points $x_{1},\ldots ,x_{6}$ form the vertices of a nonregular convex octahedron in $\mathbb{R}^3$. 
We set $L=\{ x_1 ,x_2 ,x_3 ,x_4 ,x_5 ,x_6 \}$. 
For any $\varepsilon\geq 0$, 
we define a map $d_{\varepsilon}:L\times L\to\lbrack 0,\infty )$ by 
\begin{equation}\label{L-distance-def}
d_{\varepsilon} (a,b)
=
\begin{cases}
\|a-b\| +\varepsilon ,\quad\textrm{if }\{ a,b\}=\{ x_5 ,x_6\} ,\\
\|a-b\| ,\quad\textrm{if }\{ a,b\}\neq\{ x_5 ,x_6\} ,
\end{cases}
\end{equation}
where $\| a-b\|$ is the Euclidean norm of $a-b$. 

\begin{theorem}[Lebedeva]\label{Lebedeva-th}
Fix six distinct points $x_1 ,\ldots ,x_6 \in\mathbb{R}^3$ that satisfy \eqref{lebedeva-condition1} and \eqref{lebedeva-condition2}. 
Set $L=\{ x_1 ,\ldots ,x_6 \}$. 
Then there exists $C\in (0,\infty )$ such that 
for any $\varepsilon\in (0,C\rbrack$, 
$(L,d_{\varepsilon})$ defined by \eqref{L-distance-def} is a metric space that satisfies 
the $(2+2)$-point comparison and the $(4+2)$-point comparison, 
but does not admit an isometric embedding into any $\mathrm{CAT}(0)$ space. 
\end{theorem}

A metric space $X$ is said to satisfy the {\em $(2+2)$-comparison} if for any points $x,y,z,w\in X$ and 
any points $\tilde{x},\tilde{y},\tilde{z},\tilde{w}$ in the Euclidean plane $\mathbb{R}^2$ with 
\begin{align*}
&d_{X}(x,y)=\|\tilde{x}-\tilde{y}\| ,\quad
d_{X}(y,z)=\|\tilde{y}-\tilde{z}\| , \quad
d_{X}(z,w)=\|\tilde{z}-\tilde{w}\| ,\\
&d_{X}(w,x)=\|\tilde{w}-\tilde{x}\| ,\quad
d_{X}(x,z)=\|\tilde{x}-\tilde{z}\| ,\quad
\end{align*}
we have $d_{X}(y,w)\leq\|\tilde{y}-\tilde{p}\| +\|\tilde{p}-\tilde{w}\|$ for any $\tilde{p}\in\lbrack\tilde{x},\tilde{z}\rbrack$. 
For the definition of the $(2n+2)$-point comparison for an integer $n\geq 2$, see \cite[\S 6.2]{AKP}. 

It is known that a metric space satisfies the $(2+2)$-comparison if and only if it satisfies the $\boxtimes$-inequalities. 
Although this fact is well known to experts, we include a proof for the reader's convenience, 
by combining several results that have already appeared in the literature. 
The argument is slightly indirect because we rely on previously published results.
For this purpose, we begin by recalling the definition and some properties of 
the $\mathrm{Cycl}_4(0)$ condition introduced by Gromov \cite{Gr2}, which is closely related to the $(2+2)$-comparison.
A metric space $X$ is said to satisfy the {\em $\mathrm{Cycl}_4 (0)$ condition} if and only if for any points $x,y,z,w\in X$, there 
exist points $\tilde{x},\tilde{y},\tilde{z},\tilde{w}$ in the Euclidean plane $\mathbb{R}^2$ with 
\begin{align*}
&\|\tilde{x}-\tilde{y}\|\leq d_{X}(x,y),\quad
\|\tilde{y}-\tilde{z}\|\leq d_{X}(y,z) , \quad
\|\tilde{z}-\tilde{w}\|\leq d_{X}(z,w) ,\\
&\|\tilde{w}-\tilde{x}\|\leq d_{X}(w,x) ,\quad
\|\tilde{x}-\tilde{z}\|\geq d_{X}(x,z) ,\quad
\|\tilde{y}-\tilde{w}\|\geq d_{X}(y,w).
\end{align*}
We may assume that the points $x,y,z,w\in X$ in the above definition of $\mathrm{Cycl}_4 (0)$ condition are distinct 
because if they are not distinct, $\{ x,y,z,w\}$ consists of at most three points in a metric space and clearly can embed isometrically into the Euclidean plane. 
Gromov \cite{Gr2} established the following fact. 
\begin{theorem}[Gromov \cite{Gr2}]\label{box-cycl4-th}
A metric space $X$ satisfies the $\boxtimes$-inequalities if and only if $X$ satisfies the $\mathrm{Cycl}_4 (0)$ condition. 
\end{theorem}
For the detailed proof of Theorem \ref{box-cycl4-th}, see \cite[Lemma 2.6]{KTU} and \cite[\S 6]{toyoda-cycle}. 
The following fact is standard, and a proof can be found, for example, in \cite[Lemma~1.13]{toyoda-cycle}. 
\begin{proposition}[cf. Lemma~1.13 of \cite{toyoda-cycle}]\label{cycl4-2plus2-prop}
If a metric space $X$ satisfies the $\mathrm{Cycl}_4 (0)$ condition, then $X$ satisfies the $(2+2)$-comparison. 
\end{proposition}
We are now ready to prove the announced fact. 
For any $x\in\mathbb{R}^2$ and $y,z\in\mathbb{R}^2 \setminus\{ x\}$, we denote by $\measuredangle yxz\in\lbrack 0,\pi\rbrack$ the interior angle measure at $x$ of 
the (possibly degenerate) triangle with vertices $x$, $y$ and $z$. 
\begin{proposition}\label{2plus2-boxtimes-prop}
A metric space $X$ satisfies the $(2+2)$-comparison if and only if $X$ satisfies the $\boxtimes$-inequalities. 
\end{proposition}
\begin{proof}
The ``if'' part follows from Theorem~\ref{box-cycl4-th} and Proposition~\ref{cycl4-2plus2-prop}.
To prove the ``only if'' part, we assume that $X$ satisfies the $(2+2)$-comparison. 
By Theorem \ref{box-cycl4-th}, it suffices to prove that $X$ satisfies the $\mathrm{Cycl}_4 (0)$ condition. 
Choose four distinct points $x,y,z,w\in X$ arbitrarily. 
Clearly, there exist $\tilde{x}, \tilde{y}, \tilde{z}, \tilde{w}\in\mathbb{R}^2$ such that 
\begin{align*}
&d_{X}(x,y)=\|\tilde{x}-\tilde{y}\| ,\quad
d_{X}(y,z)=\|\tilde{y}-\tilde{z}\| , \quad
d_{X}(z,w)=\|\tilde{z}-\tilde{w}\| ,\\
&d_{X}(w,x)=\|\tilde{w}-\tilde{x}\| ,\quad
d_{X}(x,z)=\|\tilde{x}-\tilde{z}\| ,
\end{align*}
and $\tilde{y}$ and $\tilde{w}$ do not lie on the same side of the line through $\tilde{x}$ and $\tilde{z}$. 
If the line segments $\lbrack\tilde{x},\tilde{z}\rbrack$ and $\lbrack\tilde{y},\tilde{w}\rbrack$ 
share a point $\tilde{p}$, then we have 
\begin{equation*}
d_{X}(y,w)\leq\|\tilde{y}-\tilde{p}\| +\|\tilde{p}-\tilde{w}\| =\|\tilde{y}-\tilde{w}\|
\end{equation*}
since $X$ satisfies the $(2+2)$-comparison. 
If $\lbrack\tilde{x},\tilde{z}\rbrack\cap\lbrack\tilde{y},\tilde{w}\rbrack =\emptyset$, then we have 
$\measuredangle\tilde{y}\tilde{x}\tilde{z}+\measuredangle\tilde{w}\tilde{x}\tilde{z}\geq\pi$ or 
$\measuredangle\tilde{y}\tilde{z}\tilde{x}+\measuredangle\tilde{w}\tilde{z}\tilde{x}\geq\pi$. 
We may assume $\measuredangle\tilde{y}\tilde{z}\tilde{x}+\measuredangle\tilde{w}\tilde{z}\tilde{x}\geq\pi$ without loss of generality. 
Then there exist $x' ,y', w' \in\mathbb{R}^2$ with 
\begin{align*}
&\|x'-y'\| =\|\tilde{x}-\tilde{y}\| =d_{X}(x,y),\quad
\|x'-w'\| =\|\tilde{x}-\tilde{w}\| =d_{X}(x,w),\\
&\|y'-w'\| =\|\tilde{y}-\tilde{z}\| +\|\tilde{w}-\tilde{z}\| =d_{X}(y,z)+d_{X}(w,z)\geq d_{X}(y,w)
\end{align*}
since we have $\|\tilde{y}-\tilde{z}\| +\|\tilde{w}-\tilde{z}\| \leq\|\tilde{y}-\tilde{x}\| +\|\tilde{w}-\tilde{x}\|$ in this case. 
Let $z'$ be the point on $\lbrack y',w'\rbrack$ with $\| y'-z' \| =d_{X}(y,z)$. 
Then Alexandrov's Lemma \cite[p.~25, 2.16]{BH} (see also \cite[Lemma~4.6]{toyoda-cycle}) implies that 
\begin{equation*}
\| x' -z' \|\geq\|\tilde{x}-\tilde{z}\| =d_{X}(x,z).  
\end{equation*}
This proves that $X$ satisfies the $\mathrm{Cycl}_4 (0)$ condition, and completes the proof.  
\end{proof}
\begin{remark}
Alexandrov's Lemma \cite[p.~25, 2.16]{BH} is stated for the case when $\tilde{y}$ and $\tilde{w}$ lie on opposite sides of the line through $\tilde{x}$ and $\tilde{z}$.  
However, it is straightforward to verify that the same conclusion also holds when $\tilde{y}$ or $\tilde{w}$ lies on the line through $\tilde{x}$ and $\tilde{z}$.  
We also note that \cite[Lemma~4.6]{toyoda-cycle} provides a generalization of Alexandrov's Lemma, formulated in a way that can be directly applied in the above proof.  
\end{remark}

\section{Proof of Theorem~\ref{main-th}}\label{proof-sec}

In this section, we prove the following theorem, which implies Theorem~\ref{main-th}. 
For any subset $M$ of $\mathbb{R}^3$ and any $R\in (0,\infty )$, we denote by $B(M,R)$ the open $R$-neighborhood of $M$ in $\mathbb{R}^3$. 
For any $x\in\mathbb{R}^3$, we denote $B(\{x\},R)$ simply by $B(x,R)$. 

\begin{theorem}\label{main-ANN-ineq-th}
Fix six distinct points $x_1 ,\ldots ,x_6 \in\mathbb{R}^3$ that satisfy \eqref{lebedeva-condition1} and \eqref{lebedeva-condition2}. 
Set $L=\{ x_1 ,\ldots ,x_6 \}$. 
Then there exists $C\in (0,\infty )$ such that 
for any $\varepsilon\in (0,C\rbrack$, 
$(L,d_{\varepsilon})$ defined by \eqref{L-distance-def} satisfies the $\mathrm{ANN}(n)$ inequalities for 
all positive integers $n$. 
\end{theorem}

\begin{proof}
We first define four positive real constants $h$, $H$, $\theta$ and $\delta$ that depend only on the choice of $x_1 ,\ldots ,x_6 \in\mathbb{R}^3$. 
We set 
\begin{align*}
h&=\min\left\{\min_{y\in\mathrm{conv}\left( L\setminus\left\{ x_5 \right\}\right)}\| x_5 -y\| ,\quad
\min_{y\in\mathrm{conv}\left( L\setminus\left\{ x_6 \right\}\right)}\| x_6 -y\| \right\} ,\\
H&=\max\left\{\max_{y\in\mathrm{conv}\left( L\setminus\left\{ x_5 \right\}\right)}\| x_5 -y\| ,\quad
\max_{y\in\mathrm{conv}\left( L\setminus\left\{ x_6 \right\}\right)}\| x_6 -y\| \right\}, \\
\theta &=\min\left\{\min_{k\in\lbrack 4\rbrack}\measuredangle x_k x_5 x_6 ,\quad
\min_{k\in\lbrack 4\rbrack}\measuredangle x_k x_6 x_5 \right\} ,
\end{align*}
where we denote by $\measuredangle x_k x_i x_j \in\lbrack 0,\pi\rbrack$ 
the interior angle measure at $x_i$ of the triangle with vertices $x_k$, $x_i$ and $x_j$. 
It follows from \eqref{lebedeva-condition1} and \eqref{lebedeva-condition2} that 
$h$, $H$ and $\theta$ are finite positive real numbers, and  
\begin{equation}\label{x5x6geq2h-ineq}
\| x_5 -x_6 \| =\| x_5 -y_0 \| +\| y_0 -x_6\| \geq 2h,
\end{equation}
where $y_0$ is the unique point in 
$(x_5 ,x_6)\cap\left(\mathrm{conv}(\{ x_1 ,x_2 ,x_3 ,x_4\})\setminus\cup_{\{ i,j\}\in\binom{\lbrack 4\rbrack}{2}}\lbrack x_i ,x_j\rbrack\right)$. 
It also follows from \eqref{lebedeva-condition1} and \eqref{lebedeva-condition2} that 
we can choose $\delta\in (0,h/2 )$ 
such that 
\begin{align}
B(\lbrack x_5 ,x_6 \rbrack ,\delta)\cap B(\lbrack x_i ,x_j \rbrack ,\delta) 
&=\emptyset ,\label{delta-def-ij}\\
B(\lbrack x_5 ,x_6 \rbrack ,\delta)\cap B(\lbrack x_5 ,x_i \rbrack ,\delta)
&\subseteq
B\left(x_5 ,\frac{h}{2}\right) ,\label{delta-def-5i}\\
B(\lbrack x_5 ,x_6 \rbrack ,\delta)\cap B(\lbrack x_6 ,x_i \rbrack ,\delta)
&\subseteq
B\left(x_6 ,\frac{h}{2}\right) \label{delta-def-6i}
\end{align}
for any $i,j\in\lbrack 4\rbrack$. 

Fix nonnegative real numbers $p_{1},\ldots ,p_{6},q_{1},\ldots ,q_{6}$ and 
a $6\times 6$ matrix 
$(\pi_{ij})$ with nonnegative real entries that satisfy 
\begin{equation}\label{pqpi6-condition}
\sum_{k=1}^{6}p_k =\sum_{k=1}^{6}q_k =1,\quad
\sum_{k=1}^{6}\pi_{ik}=p_i +q_i
\end{equation}
for every $i\in\lbrack 6\rbrack$. 
By Proposition~\ref{n-points-prop}, to prove the theorem, it suffices to show that 
there exists a constant $C>0$, depending only on the choice of $x_1 ,\ldots ,x_6 \in\mathbb{R}^3$, 
such that the inequality 
\begin{equation}\label{main-ANN-ineq-th-desired-ineq}
\frac{1}{2}\sum_{i,j\in\lbrack 6\rbrack : \pi_{ij}+\pi_{ji}> 0}
\frac{\pi_{ij}\pi_{ji}}{\pi_{ij}+\pi_{ji}}d_{\varepsilon}(x_i ,x_j)^2
\leq
\sum_{i=1}^{6}\sum_{j=1}^{6}p_i q_j d_{\varepsilon}(x_i ,x_j)^2
\end{equation}
holds for any $\varepsilon\in (0,C\rbrack$. 
Since \eqref{main-ANN-ineq-th-desired-ineq} holds true for every $\varepsilon\in\lbrack 0,\infty )$ 
whenever $\pi_{56}=0$ or $\pi_{65}=0$ by definition of $d_{\varepsilon}$, 
we assume that $\pi_{56}>0$ and $\pi_{65}>0$.

For each $\varepsilon\geq 0$, we define $F(\varepsilon )\in\mathbb{R}$ by 
\begin{equation*}
F(\varepsilon )
=
\sum_{i=1}^{6}\sum_{j=1}^{6}p_i q_j d_{\varepsilon}(x_i ,x_j)^2
-
\frac{1}{2}\sum_{i,j\in\lbrack 6\rbrack : \pi_{ij}+\pi_{ji}>0}
\frac{\pi_{ij}\pi_{ji}}{\pi_{ij}+\pi_{ji}}d_{\varepsilon}(x_i ,x_j)^2 .
\end{equation*}
Then it follows from the definition of $d_{\varepsilon}$ that 
\begin{equation}\label{main-ANN-ineq-th-F-ineq1}
F(\varepsilon )
\geq
\left( p_5 q_6 +p_6 q_5 -\frac{\pi_{56}\pi_{65}}{\pi_{56}+\pi_{65}}\right) f(\varepsilon ) +F(0),
\end{equation}
where $f:\lbrack 0,\infty )\to\lbrack 0,\infty )$ is the strictly increasing function defined by 
\begin{equation*}
f(\varepsilon )=
\left(\| x_{5} -x_{6} \| +\varepsilon\right)^2 -\| x_{5}-x_{6}\|^2 ,
\quad \varepsilon\in\lbrack 0,\infty ). 
\end{equation*}
We define $z_p \in\mathbb{R}^3$, $z_q \in\mathbb{R}^3$ 
and $z_{ij}\in\mathbb{R}^3$ for any $i,j\in\lbrack 6\rbrack$ with $\pi_{ij}+\pi_{ji}> 0$ by 
\begin{equation*}
z_p=\sum_{k=1}^6 p_k x_k ,\quad
z_q=\sum_{k=1}^6 q_k x_k ,\quad
z_{ij}=\frac{\pi_{ij}x_i +\pi_{ji}x_j}{\pi_{ij}+\pi_{ji}}.
\end{equation*}
In other words, $z_p$, $z_q$ and $z_{ij}$ are the barycenters of the probability measures $\sum_{k=1}^{6} p_k \delta_{x_k}$, 
$\sum_{k=1}^{6} q_k \delta_{x_k}$ and $(\pi_{ij}\delta_{x_i}+\pi_{ji}\delta_{x_j})/(\pi_{ij}+\pi_{ji})$, respectively. 
By Proposition~\ref{pi-prop}, we have 
\begin{align*}
F(0)
\geq
\sum_{i,j\in\lbrack 6\rbrack : \pi_{ij}+\pi_{ji}> 0}\pi_{ij}\| z_p -z_{ij}\|^2 ,\qquad
F(0)
\geq
\sum_{i,j\in\lbrack 6\rbrack : \pi_{ij}+\pi_{ji}> 0}\pi_{ij}\| z_q -z_{ij}\|^2 .
\end{align*}
Combining these inequalities with \eqref{main-ANN-ineq-th-F-ineq1}, we obtain 
\begin{align}
F(\varepsilon)
\geq
&\left( p_5 q_6 +p_6 q_5 -\frac{\pi_{56}\pi_{65}}{\pi_{56}+\pi_{65}}\right) f(\varepsilon )\nonumber\\
&\quad
+
\max\left\{\sum_{i,j\in\lbrack 6\rbrack : \pi_{ij}+\pi_{ji}> 0}\pi_{ij}\| z_{p}-z_{ij}\|^2 ,
\sum_{i,j\in\lbrack 6\rbrack : \pi_{ij}+\pi_{ji}> 0}\pi_{ij}\| z_{q}-z_{ij}\|^2 \right\} .
\label{main-ANN-ineq-th-F-ineq}
\end{align}
Set $S=\sum_{i=1}^{4}\sum_{j=1}^{6}\pi_{ij}$. 
Then we have 
\begin{equation*}
p_6 \geq 1-p_{5}-S,\quad
q_6 \geq 1-q_{5}-S,\quad
\pi_{56}\leq p_{5}+q_{5},\quad
\pi_{65}\leq 2-p_{5}-q_{5}, 
\end{equation*}
and therefore 
\begin{align}
p_{5}q_{6}+p_{6}q_{5}-\frac{\pi_{56}\pi_{65}}{\pi_{56}+\pi_{65}}
&\geq
p_{5}(1-q_{5}-S)+(1-p_{5}-S)q_{5}-\frac{(p_{5}+q_{5})(2-p_{5}-q_{5})}{2} \nonumber\\
&=
\frac{(p_{5}-q_{5})^2}{2}-(p_{5}+q_{5})S
\geq
-(p_{5}+q_{5})S
\label{S-ineq}
\end{align}
since the function $\varphi (x,y)=xy/(x+y)$, $x,y\in(0,\infty )$ is increasing with respect to both $x$ and $y$. 
Similarly, we also have 
\begin{equation*}
p_{5}q_{6}+p_{6}q_{5}-\frac{\pi_{56}\pi_{65}}{\pi_{56}+\pi_{65}}
\geq
-(p_{6}+q_{6})S. 
\end{equation*}

We consider three cases. 

\textsc{Case 1}: 
{\em The inequality $\| z_p -z_{56}\|\geq\delta$ holds, or the inequality $\| z_q -z_{56}\|\geq\delta$ holds.} 
In this case, it follows from \eqref{main-ANN-ineq-th-F-ineq} that 
\begin{equation*}
F(\varepsilon )
\geq
-\frac{\pi_{56}\pi_{65}}{\pi_{56}+\pi_{65}}f(\varepsilon ) +(\pi_{56}+\pi_{65})\delta^2
\geq
(\pi_{56}+\pi_{65})(\delta^2 -f(\varepsilon )).
\end{equation*}
Therefore, we have $F(\varepsilon )\geq 0$ for any $\varepsilon\in (0,f^{-1}(\delta^2 )\rbrack$. 

\textsc{Case 2}: 
{\em The inequality $\| z_p -z_{ij}\|\geq\delta$ holds for all $i,j\in\lbrack 6\rbrack$ 
with $\{ i,j\}\not\subseteq\{ 5,6\}$ and $\pi_{ij}+\pi_{ji}>0$, or the inequality $\| z_q -z_{ij}\|\geq\delta$ holds for all $i,j\in\lbrack 6\rbrack$ 
with $\{ i,j\}\not\subseteq\{ 5,6\}$ and $\pi_{ij}+\pi_{ji}>0$.} 
In this case, it follows from \eqref{main-ANN-ineq-th-F-ineq} and \eqref{S-ineq} that 
\begin{align*}
F(\varepsilon )
\geq
-(p_{5}+q_{5})S f(\varepsilon ) +S \delta^2
\geq
S \left( \delta^{2} -2f(\varepsilon )\right) .
\end{align*}
Therefore, we have $F(\varepsilon )\geq 0$ for any $\varepsilon\in (0,f^{-1}(\delta^2 /2)\rbrack$. 

\textsc{Case 3}: {\em Neither \textsc{Case~1} nor \textsc{Case~2} holds.} 
In this case, 
there exist $i_{0},j_{0},i_{1}, j_{1}\in\lbrack 6\rbrack$ such that 
\begin{align*}
&\{ i_{0},j_{0}\}\not\subseteq\{ 5,6\} ,\quad
\{ i_{1},j_{1}\}\not\subseteq\{ 5,6\} ,\\
&z_{p}\in B\left(\lbrack x_5 ,x_6 \rbrack , \delta\right)
\cap B\left(\lbrack x_{i_{0}} ,x_{j_{0}} \rbrack , \delta\right) ,\quad
z_{q}\in B\left(\lbrack x_5 ,x_6 \rbrack , \delta\right)
\cap B\left(\lbrack x_{i_{1}} ,x_{j_{1}} \rbrack , \delta\right) .
\end{align*}
It follows from \eqref{delta-def-ij} that 
\begin{equation*}
\{ i_{0},j_{0}\}\cap\{ 5,6\}\neq\emptyset ,\quad
\{ i_{1},j_{1}\}\cap\{ 5,6\}\neq\emptyset .
\end{equation*}
Therefore, it follows from \eqref{delta-def-5i} and \eqref{delta-def-6i} that 
\begin{equation*}
\{z_p ,z_q\}\subseteq B\left(x_5 ,\frac{h}{2}\right)\cup B\left(x_6 ,\frac{h}{2}\right) .
\end{equation*}
If we had 
\begin{equation*}
z_{p}\in B\left(x_5 ,\frac{h}{2}\right) ,\quad
z_{q}\in B\left(x_6 ,\frac{h}{2}\right) ,
\end{equation*}
then we would have 
\begin{align*}
\| x_{5}-x_{6}\|
&\leq
\| x_{5}-z_{p}\| +\| z_{p}-z_{56}\| +\| z_{56}-z_{q}\| +\| z_{q}-x_{6}\| \\
&\leq
\frac{h}{2}+\delta +\delta +\frac{h}{2}
<
2h ,
\end{align*}
contradicting \eqref{x5x6geq2h-ineq}.
Similarly, it is also impossible to have 
\begin{equation*}
z_{p}\in B\left(x_6 ,\frac{h}{2}\right) ,\quad
z_{q}\in B\left(x_5 ,\frac{h}{2}\right) .
\end{equation*}
Therefore, we have 
\begin{equation}\label{z-near-x5}
\{z_p ,z_q\}\subseteq B\left(x_5 ,\frac{h}{2}\right)
\end{equation}
or 
\begin{equation}\label{z-near-x6}
\{z_p ,z_q\}\subseteq B\left(x_6 ,\frac{h}{2}\right) .
\end{equation}
Because it is easily seen that we can obtain the same estimate of $F(\varepsilon )$ in the same way as we describe below 
whether \eqref{z-near-x5} holds or \eqref{z-near-x6} holds, 
we assume that \eqref{z-near-x6} holds. 

Set 
\begin{equation*}
t_i =\frac{\pi_{i6}}{\pi_{i6}+\pi_{6i}}
\end{equation*}
for each $i\in\lbrack 5\rbrack$ with $\pi_{i6}+\pi_{6i}>0$. 
Note that since we have assumed $\pi_{56}>0$, it follows that $t_{5}>0$. 
Let $K_0$ be the nonnegative real number that satisfies 
\begin{equation*}
\min\left\{ p_6 ,q_6 \right\}
=
1-K_0 t_5 .
\end{equation*}
Then we have 
\begin{multline}
p_5 q_6 +p_6 q_5
-\frac{\pi_{56}\pi_{65}}{\pi_{56}+\pi_{65}}
\geq
(p_5 +q_5 )\min\left\{ p_6 ,q_6 \right\} -t_{5}(1-t_{5})(\pi_{56}+\pi_{65}) \\
\geq
\pi_{56}(1-K_0 t_{5}) -t_{5}(1-t_{5})(\pi_{56}+\pi_{65})
=
(\pi_{56}+\pi_{65})(1-K_{0})t_5^2 .
\label{K0-estimate}
\end{multline}
Since \eqref{main-ANN-ineq-th-F-ineq} and \eqref{K0-estimate} 
imply that the inequality $F(\varepsilon )\geq 0$ holds for any $\varepsilon >0$ whenever $K_0 \leq 1$, 
we assume that $K_0 >1$. 
We define $z\in\mathbb{R}^3$ and $\tilde{z}\in\mathbb{R}^3$ by 
\begin{equation*}
z
=
\begin{cases}
z_p ,\quad\textrm{if }p_6 \leq q_6 ,\\
z_q ,\quad\textrm{if }q_6 < p_6 ,
\end{cases}
\quad
\tilde{z}
=
\begin{cases}
\frac{1}{1-p_{6}}\sum_{k=1}^5 p_k x_k ,\quad\textrm{if }p_6 \leq q_6 ,\\
\frac{1}{1-q_{6}}\sum_{k=1}^5 q_k x_k ,\quad\textrm{if }q_6 < p_6 .
\end{cases}
\end{equation*}
To complete the proof, 
we fix a constant $\gamma\in (0,\infty )$ arbitrarily, and consider three subcases. 

\textsc{Subcase 3a}: 
{\em The inequality $K_0 \geq (1+\gamma )\| x_5 -x_6 \| /h$ holds.}
In this case, we have 
\begin{align*}
\| z -z_{56}\|
&\geq
\| z -x_{6}\| -\| z_{56} -x_{6}\| \\
&=
K_{0}t_{5}\left\|\tilde{z}-x_6 \right\| -t_{5}\| x_5 -x_6 \|
\geq
\left( K_0 h-\| x_5 -x_6 \|\right) t_{5}>0.
\end{align*}
Together with \eqref{main-ANN-ineq-th-F-ineq} and \eqref{K0-estimate}, this implies that 
\begin{equation*}
F(\varepsilon )
\geq
-(\pi_{56}+\pi_{65})(K_{0}-1)t_{5}^2 f(\varepsilon )
+(\pi_{56}+\pi_{65})\left( K_0 h-\| x_5 -x_6 \|\right)^2 t_5^2 .
\end{equation*}
It follows that we have $F(\varepsilon )\geq 0$ whenever 
\begin{equation*}
f(\varepsilon )
\leq
\frac{(K_0 h-\| x_5 -x_6 \| )^2}{K_0 -1}.
\end{equation*}
Since $\| x_5 -x_6 \| /h \geq 1$, it is easily seen that the function 
\begin{equation*}
\varphi_0 (K)=\frac{(K h-\| x_5 -x_6 \| )^2}{K -1},\quad K \in\left\lbrack\frac{(1+\gamma )\| x_5 -x_6 \|}{h},\infty\right)
\end{equation*}
is increasing. 
Therefore, we have $F(\varepsilon )\geq 0$ whenever 
\begin{equation*}
\varepsilon\leq f^{-1}\left(\frac{h\gamma^2 \| x_5-x_6 \|^2}{(1+\gamma )\| x_5-x_6 \| -h}\right) .
\end{equation*}

\textsc{Subcase 3b}: 
{\em The inequalities $1<K_0 <(1+\gamma )\| x_5 -x_6 \| /h$ and $\measuredangle zx_6 x_5 \geq\theta /2$ hold.} 
In this case, there exists a point $w\in\lbrack z,z_{56}\rbrack$ with $\measuredangle w x_{6}z_{56}=\theta /2$. 
Let $w'$ be the foot of the perpendicular from $z_{56}$ to the line through $x_{6}$ and $w$. 
Then we have 
\begin{align*}
\| z-z_{56}\|
&=
\| z-w\| +\| w-z_{56}\| \\
&\geq
\| w' -z_{56}\|
=
\left(\sin\frac{\theta}{2}\right)\| z_{56}-x_{6}\|
=
\left(\sin\frac{\theta}{2}\right)\| x_{5}-x_{6}\| t_{5}.
\end{align*}
Together with \eqref{main-ANN-ineq-th-F-ineq} and \eqref{K0-estimate}, this implies that 
\begin{equation*}
F(\varepsilon )
\geq
(\pi_{56}+\pi_{65})
\left( -\left( \frac{(1+\gamma )\| x_5 -x_6 \|}{h}-1\right) f(\varepsilon )
+\left(\sin\frac{\theta}{2}\right)^2 \| x_{5}-x_{6}\|^2 \right)t_5^2 .
\end{equation*}
Therefore, we have $F(\varepsilon )\geq 0$ whenever 
\begin{equation*}
\varepsilon
\leq
f^{-1}\left(\frac{h\left(\sin\frac{\theta}{2}\right)^2 \| x_5 -x_6 \|^2 }{(1+\gamma )\| x_5 -x_6 \| -h}\right) .
\end{equation*}

\textsc{Subcase 3c}: 
{\em The inequalities $1<K_0 <(1+\gamma )\| x_5 -x_6 \| /h$ and $\measuredangle zx_6 x_5 <\theta /2$ hold.} 
In this case, by definition of $\theta$, for every $m\in\lbrack 4\rbrack$ with $\pi_{m6}>0$, we have 
$\measuredangle zx_6 z_{m6}=\measuredangle zx_6 x_m >\theta /2$. 
Hence, for each such $m$, 
there exists a point $w_{m}\in\lbrack z,z_{m6}\rbrack$ satisfying $\measuredangle z x_{6}w_{m} =\theta /2$. 
Let $w'_{m}$ denote the foot of the perpendicular from $z$ to the line through $x_{6}$ and $w_{m}$. 
Then we have 
\begin{align}
\| z-z_{m6}\|
&=
\| z-w_{m}\| +\| w_{m}-z_{m6}\| 
\geq
\| z-w'_{m}\| 
=
\left(\sin\frac{\theta}{2}\right)\| z-x_{6}\| \nonumber\\
&=
\left(\sin\frac{\theta}{2}\right)\|\tilde{z}-x_{6}\|K_0  t_{5}
\geq
\left(\sin\frac{\theta}{2}\right)h t_{5}
\label{z-zk6-ineqs}
\end{align}
for every $m\in\lbrack 4\rbrack$ with $\pi_{m6}>0$. 
We now prove that the inequality 
\begin{equation}\label{key-ineq-in-subcase3c}
(\pi_{m6}+\pi_{6m})\| z-z_{m6}\|^2
\geq
\min\left\{
\frac{h^4 \left(\sin\frac{\theta}{2}\right)^2}{(1+\gamma )^2 \| x_5 -x_6 \| H}, 
\gamma^2 \| x_{5}-x_{6}\| H 
\right\}
\pi_{m6}t_{5}
\end{equation}
holds for all $m\in\lbrack 4\rbrack$ with $\pi_{m6}>0$ by considering two separate cases. 
We first consider the case where $m\in\lbrack 4\rbrack$ with $\pi_{m6}>0$ satisfies 
\begin{equation*}
t_{m}\| x_{m}-x_{6}\|
\leq
\frac{(1+\gamma )^2 \| x_5 -x_6 \| H t_{5}}{h} .
\end{equation*}
In this case, we have 
\begin{equation*}
0<t_{m}
\leq
\frac{(1+\gamma )^2 \| x_5 -x_6 \| H t_{5}}{h^2}. 
\end{equation*}
Together with \eqref{z-zk6-ineqs}, this implies that 
\begin{equation}\label{pil6pi6l-ineqs}
(\pi_{m6}+\pi_{6m})\| z-z_{m6}\|^2
=
\frac{1}{t_m}\pi_{m6}\| z-z_{m6}\|^2
\geq
\frac{h^4 \left(\sin\frac{\theta}{2}\right)^2}{(1+\gamma )^2 \| x_5 -x_6 \| H}\pi_{m6}t_{5}. 
\end{equation}
We next consider the case where $m\in\lbrack 4\rbrack$ with $\pi_{m6}>0$ satisfies 
\begin{equation*}
t_{m}\| x_{m}-x_{6}\|
>
\frac{(1+\gamma )^2 \| x_5 -x_6 \| H t_{5}}{h}.
\end{equation*}
In this case, we have 
\begin{align*}
\| z-z_{m6}\|
&\geq
\| z_{m6}-x_{6}\| -\| z-x_{6}\|
=
t_{m}\| x_{m}-x_{6}\| -K_{0}t_{5}\|\tilde{z}-x_{6}\| \\
&>
t_{m}\| x_{m}-x_{6}\| -\frac{(1+\gamma )\| x_{5}-x_{6}\| H t_{5}}{h}
>
\frac{\gamma (1+\gamma )\| x_{5}-x_{6}\| H t_5}{h}
>0,
\end{align*}
and define $K_{m}\in (\gamma (1+\gamma )\| x_{5}-x_{6}\| H/h,\infty )$ to be the positive real number that satisfies 
\begin{equation*}
K_{m}t_{5}
=
t_{m}\| x_{m}-x_{6}\| -\frac{(1+\gamma )\| x_{5}-x_{6}\| H t_{5}}{h}.
\end{equation*}
Then we have 
\begin{equation*}
\| z-z_{m6}\|
>
K_{m}t_{5},\quad
\pi_{m6}+\pi_{6m}
=
\frac{1}{t_m}\pi_{m6}
\geq
\frac{h^2}{\left(hK_{m}+(1+\gamma )\| x_{5}-x_{6}\| H\right) t_{5}}\pi_{m6} ,
\end{equation*}
and therefore 
\begin{equation*}
(\pi_{m6}+\pi_{6m})\| z-z_{m6}\|^2
>
\frac{h^2 K_{m}^2}{hK_{m}+(1+\gamma )\| x_{5}-x_{6}\| H}\pi_{m6}t_{5}.
\end{equation*}
Because the function 
\begin{equation*}
\varphi_1 (K)
=
\frac{h^2 K^2}{hK+(1+\gamma )\| x_{5}-x_{6}\| H},\quad
K\in\left\lbrack\frac{\gamma (1+\gamma )\| x_5 -x_6\| H}{h},\infty\right)
\end{equation*}
is increasing, we obtain 
\begin{equation*}
(\pi_{m6}+\pi_{6m})\| z-z_{m6}\|^2
>
\gamma^2 \| x_{5}-x_{6}\| H \pi_{m6}t_{5} .
\end{equation*}
Together with \eqref{pil6pi6l-ineqs}, 
this proves that \eqref{key-ineq-in-subcase3c} holds 
for all $m\in\lbrack 4\rbrack$ with $\pi_{m6}>0$. 
For any $i\in\lbrack 4\rbrack$ and any $j\in\lbrack 5\rbrack$ with $\pi_{ij}+\pi_{ji}>0$, we have 
\begin{equation*}
\| z_p - z_{ij}\| >\frac{h}{2},\quad
\| z_q - z_{ij}\| >\frac{h}{2}
\end{equation*}
by \eqref{z-near-x6}, and therefore 
\begin{equation}\label{z-zij-S1-ineq}
\pi_{ij}\| z-z_{ij}\|^2 \geq\frac{h^2}{4}\pi_{ij}\geq\frac{h^2}{4}\pi_{ij}t_{5}.
\end{equation}
Set 
\begin{equation*}
c
=
\min\left\{\frac{h^4 \left(\sin\frac{\theta}{2}\right)^2}{(1+\gamma )^2 \| x_5 -x_6 \| H},\quad\gamma^2 \| x_{5}-x_{6}\| H,\quad\frac{h^2}{4}\right\} .
\end{equation*}
Then it follows from \eqref{key-ineq-in-subcase3c} and \eqref{z-zij-S1-ineq} that 
\begin{align}
&\sum_{i,j\in\lbrack 6\rbrack : \pi_{ij}+\pi_{ji}> 0}\pi_{ij}\| z-z_{ij}\|^2 \nonumber\\
&\qquad\geq
\sum_{i\in\lbrack 4\rbrack,j\in\lbrack 5\rbrack : \pi_{ij}+\pi_{ji}> 0}\pi_{ij}\| z-z_{ij}\|^2
+
\sum_{m\in\lbrack 4\rbrack : \pi_{m6}> 0}(\pi_{m6}+\pi_{6m}) \| z-z_{m6}\|^2 \nonumber\\
&\qquad\geq
\sum_{i\in\lbrack 4\rbrack,j\in\lbrack 5\rbrack : \pi_{ij}+\pi_{ji}> 0}c\pi_{ij}t_{5}
+
\sum_{m\in\lbrack 4\rbrack : \pi_{m6}> 0}c\pi_{m6}t_{5}
=
c S t_{5},\label{S1ct6-ineq}
\end{align}
where $S=\sum_{i=1}^{4}\sum_{j=1}^{6}\pi_{ij}$ as we have defined before. 
On the other hand, we have 
\begin{align}
&p_5 q_6 +q_5 p_6 -\frac{\pi_{56}\pi_{65}}{\pi_{56}+\pi_{65}}
\geq
-(p_{5}+q_{5})S \nonumber\\
&\qquad
\geq
-(2-p_{6}-q_{6})S
\geq
-2K_{0}S t_{5}
\geq
-\frac{2(1+\gamma )\| x_5 -x_6 \|}{h}S t_{5}
\label{pqqppipi-S-ineq}
\end{align}
by \eqref{S-ineq}. 
It follows from \eqref{main-ANN-ineq-th-F-ineq}, \eqref{S1ct6-ineq} and \eqref{pqqppipi-S-ineq} that 
\begin{equation*}
F(\varepsilon )
\geq
-\frac{2(1+\gamma )\| x_5 -x_6 \| }{h}S t_{5}f(\varepsilon ) +c S t_{5}.
\end{equation*}
Therefore we have $F(\varepsilon )\geq 0$ whenever 
\begin{equation*}
\varepsilon
\leq
f^{-1}\left(\frac{c h}{2(1+\gamma )\| x_5 -x_6 \|}\right) .
\end{equation*}

Set 
\begin{align*}
C_{\gamma}=
\min\bigg\{
f^{-1}(\delta^2 /2) ,\hspace{1mm}
&f^{-1}\left(\frac{h\gamma^2 \| x_5-x_6 \|^2}{(1+\gamma )\| x_5-x_6 \| -h}\right) ,\\
&\quad f^{-1}\left(\frac{h\left(\sin\frac{\theta}{2}\right)^2 \| x_5 -x_6 \|^2 }{(1+\gamma )\| x_5 -x_6 \| -h}\right) ,\hspace{1mm}
f^{-1}\left(\frac{c h}{2(1+\gamma )\| x_5 -x_6 \|}\right)
\bigg\} .
\end{align*}
We have proved so far that for any $\gamma\in (0,\infty )$, 
the inequality \eqref{main-ANN-ineq-th-desired-ineq} holds for any $\varepsilon\in (0,C_{\gamma}\rbrack$. 
The constant $C_{\gamma}\in (0,\infty )$ depends only on the choice of $x_1 ,\ldots ,x_6 \in\mathbb{R}^3$ and on $\gamma\in (0,\infty )$. 
Since we can fix $\gamma\in (0,\infty )$ arbitrarily, this completes the proof. 
\end{proof}

Theorem~\ref{main-th} and Corollary~\ref{main-coro} follow from Theorem~\ref{main-ANN-ineq-th} immediately. 

\begin{proof}[Proof of Theorem~\ref{main-th}]
Theorem~\ref{Lebedeva-th} and Theorem~\ref{main-ANN-ineq-th} imply Theorem~\ref{main-th}.
\end{proof}

\begin{proof}[Proof of Corollary~\ref{main-coro}]
Theorem~\ref{main-th} and Corollary~\ref{pi-to-AB-coro} imply Corollary~\ref{main-coro}.
\end{proof}

\bigskip
\begin{acknowledgement}
The author thanks the anonymous referees for a careful reading and valuable comments, which helped to clarify the logical structure of the paper.
This work was supported in part by JSPS KAKENHI Grant Number JP21K03254.
\end{acknowledgement}

\end{document}